\documentclass[preprint]{article}
\pdfoutput=1 
\usepackage{amsthm,amsmath,amssymb}
\usepackage{epsfig}
\usepackage{graphicx}
\usepackage{float}
\usepackage{caption}
\usepackage{subfigure}
\usepackage{enumerate}
\usepackage{epstopdf,color}
\usepackage{multirow}
\usepackage[square,sort,comma,numbers]{natbib}
\usepackage{listings}
\usepackage{bm}
\usepackage{bbm}
\usepackage{authblk}
\usepackage[top=1in, bottom=1in, left=1.4in, right=1.4in]{geometry}
\usepackage{hyperref}
\hypersetup{
	colorlinks=true,
	linkcolor=blue,
	filecolor=blue,      
	urlcolor=blue,
	citecolor=blue,
}

\newcommand{\bbV}{{\bf V}}
\newcommand{\bbU}{{\bf U}}
\newcommand{\bbD}{{\bf D}}
\newcommand{\bbA}{{\bf A}}
\newcommand{\bbTp}{{\bf T}_p}
\newcommand{\bbX}{{\bf X}}

\newcommand{\bbZ}{{\bf Z}}
\newcommand{\bbx}{{\bf x}}

\newcommand{\bbS}{{\bf S}}
\newcommand{\bbB}{{\bf B}}
\newcommand{\bSi}{\pmb \Sigma}
\newcommand{\bbalp}{{\bf \alpha}}

\newcommand{\bgl}{{\bf \lambda}}

\newcommand{\bGma}{{\bf \Gamma}}

\newcommand{\bgO}{{\bf \Omega}}

\newcommand{\bbI}{{\bf I}}

\newcommand{\bqn}{\begin{eqnarray*}}
	\newcommand{\eqn}{\end{eqnarray*}}

\newcommand{\bqa}{\begin{eqnarray}}
	\newcommand{\eqa}{\end{eqnarray}}

\newcommand{\al}{\alpha}

\newcommand{\CYRS}{\CYRS}

\newcommand{\E}{{\mathbb E}}

\newtheorem{thm}{Theorem}[section]

\newtheorem{lemma}[thm]{Lemma}

\def\E{\mbox{E}}

\RequirePackage{amsmath,amssymb,graphicx,lscape,txfonts}
\numberwithin{equation}{section}
\numberwithin{equation}{section}
\begin{document}	
\title{ CLT for LSS of sample covariance matrices with unbounded dispersions}

\author[]{Zhijun Liu}
\author[]{Zhidong Bai}
\author[]{Jiang Hu}
\author[]{Haiyan Song}
\affil[]{
\textit{{\normalsize School of Mathematics and Statistics, Northeast Normal University}}

\textit{{\small  \href{mailto:liuzj037@nenu.edu.cn}{liuzj037@nenu.edu.cn};
\href{mailto:baizd@nenu.edu.cn}{baizd@nenu.edu.cn};
\href{mailto:huj156@nenu.edu.cn}{huj156@nenu.edu.cn};
\href{mailto:songhy716@nenu.edu.cn}{songhy716@nenu.edu.cn}}}
}

\date{}
\maketitle
\begin{abstract}
Under the high-dimensional setting that data dimension and sample size tend to infinity proportionally, we derive the central limit theorem (CLT) for linear spectral statistics (LSS) of large-dimensional sample covariance matrix. Different from existing literature, our results do not require the assumption that the population covariance matrices are bounded. Moreover, many common kernel functions in the real data such as logarithmic functions and polynomial functions are allowed in this paper. In our model, the number of spiked eigenvalues can be fixed or tend to infinity. One salient feature of the asymptotic mean and covariance in our proposed central limit theorem is that it is related to the divergence order of the population spectral norm. 
\end{abstract}	
\section{Introduction}

With the rapid development of modern technology, statisticians are confronted with the task of analyzing data with ever increasing dimension, such as all human DNA base pairs can reach 3.2 billion, therefore the data matrix they form is high-dimensional. In high-dimensional setting, many frequently used statistics such as linear spectral statistics in multivariate analysis perform in a completely different manner than they do on data of low dimension, in the introduction of \cite{10.1214/aop/1078415845}, they gave an example to illustrate the different behaviors of linear spectral statistics when the data is low-dimensional and high-dimensional. Consider sample covariance matrix $ \bbB=\dfrac{1}{n}\bbTp\bbX\bbX^{\ast}\bbTp^{\ast} $, where $ \bbX $ is a $ p\times n $ matrix with independent indentically distributed (i.i.d.) entries, $ \bbTp $ is a $p \times p$ deterministic matrix, $\bbTp\bbX$ can be seen as a random sample from the population with the population covariance matrix $\bbTp\bbTp^{\ast}=\bSi$.  It was showed in \cite{10.1214/aop/1078415845} that the linear spectral statistics of $ \bbB $ converged almost surely (a.s.) to a nonrandom quantity, and the rate of convergence is $ \dfrac{1}{n} $, and linear spectral statistics sequence form a tight sequence. Moreover, if $ Ex_{ij}^{2}=0 $, and $ E\left| x_{ij}\right| ^{4}=2 $, or if $ x_{ij} $ and $ \bSi $ are real and $ E\left| x_{ij}\right| ^{4}=3 $, they proved linear spectral statistics had Gaussian limits. They firstly introduced random matrix theory into mathematical statistics, therefore some  remarkable tools can be borrowed from random matrix theory.

Linear spectral statistics have a wide application in multivariate analysis, as pointed out in some literature, they can be a remarkable tool in many fields such as variable selection, community detection, and so on. Their asymptotic properties have been used to make statistical inference, such as hypothesis testing or parameter estimation. For example, in \cite{10.1214/09-AOS694}, they used asymptotic distribution of LSS of sample covariance matrix to modify LRT on large-dimensional covariance matrix.	Here we give some another applications of linear spectral statistics (LSS) in hypothesis testing.
For example:
\begin{itemize}
\item \textit{Wilks, likelihood ratio test}: $T_{W}=\sum_{i=1}^{p} \log \left(1+I_{i}\right)$. 
\item \textit{Lawly-Hotelling trace test}: $T_{L H}=\sum_{i=1}^{p} I_{i}$.
\item \textit{Bartlett-Nanda-Pillai trace test}: $T_{B N P}=\sum_{i=1}^{p} \frac{I_{i}}{1+I_{i}}$.
\item \textit{Roy largest root test}: $T_{R}=I_{1}$.
\end{itemize}
\textit{where $ I_{i} $ is the eigenvalue of $ \textbf{F}=\bbB_{1}\bbB_{2}^{-1} $ matrix.}

Following \cite{10.1214/aop/1078415845}'s development, there are many extensions under more relaxed settings.   The result in \cite{10.1214/aop/1078415845} was generalized by \cite{pan2008central} by removing the constraint on the fourth moment of the underlying random variables.  After that \cite{pan2014comparison} showed that the CLT of LSS for non-centered sample covariance matrices and \cite{ 10.1214/14-AOS1292} studied the unbiased sample covariance matrix when the population mean is unknown with  general moment conditions. \cite{chen2015clt} focused on the ultra-high dimensional case when the dimension $p$ is much larger than the sample size $n$. Other extensions on this topic can be found in \cite{najim2016gaussian, baik2018ferromagnetic, hu2019high, 10.3150/20-BEJ1237}, and so on.

Nowadays, it becomes more and more convenient for statisticans to collect data, and the data matrix they acquired become diversified. The main inpetus for this work comes from a common situation that the leading eigenvalues of sample covariance matrix generated by data matrix may be greater than the threshold. Almost all the references above have traditionally assumed that the spectral norms of $ \bSi $ are bounded in $ n $, and they do not take into account the case that the leading eigenvalues of  $ \bSi $ tend to infinity. It is a restriction in application in high-dimensional data because in many fields such as economics and wireless communication networks, the leading eigenvalues may tend to infinity. For example:
\begin{itemize}
\item \textit{\textbf{Panel data model}}(\cite{doi:10.1080/07474938.2017.1307580}): Many efforts have been put into analyzing the asymptotic power of sphericity test in high-dimensional setup, where the number of cross-sectional units $ n $ in a panel is large, whereas the number of time series observations $ T $ could also be large. When $ n $ tends to infinity jointly with $ T $, it will happen that the norm of perturbation term in the alternative hypothesis is greater than the threshold or even it goes to infinity. In this case, the existing methods which assume $ \bSi $ are bounded are not applicable. 
\item \textit{\textbf{Signal detection}}(\cite{10.1093/biomet/asw060}): We consider a measurement system consisting of $ m $ senmors, a standard model for the observed samples to detect a single signal is 
\begin{equation*} 
	x=\rho_{s}^{1/2}uh +\sigma\xi,   
\end{equation*} 
where $ h $ is an unknown $ m $-dimensional vector, $ u $ is a random variable distributed as $ N(0,1) $, $  \rho_{s}$ is the signal strength, $\sigma  $ is the noise level, and $ \xi $ is a random noise vector, independent of $ u $. Generally speaking, the noise level is lower, while the signal strength is larger, sometimes tend to infinity. Traditional methods which assumed bounded spectrum is excluded when statisticians confont with the task that the signal strength tend to infinity. 
\item \textit{\textbf{$ m $-factor structure}}: We focus on the properties of factors, which is the spiked eigenvalues of the population covariance matrix. Because of the complexity of the real data, it may happen that the spiked eigenvalues is very large or even goes to infinity.
\end{itemize}
For these reasons, it is essential and urgent to obtain the asymptotic properties of LSS when $ \bSi $ are unbounded.
Recently, in the work \cite{yin2021spectral},  they obtained the asymptotic distribution of the spectral statistics under some moment conditions with the assumption that the spectral norm of $ \bSi $ are unbounded, and their kernel function are polynomial functions.

This paper focuses on more general spiked covariance matrices instead of block diagonal structure, covariance matrix
$$
\bSi=\mathbf{V}\left(\begin{array}{cc}
	\bbD_{1} & 0 \\
	0 & \bbD_{2}
\end{array}\right) \mathbf{V}^{\ast}
$$
where $\mathbf{V}$ is an unitary matrix, $\bbD_{1}$ is a diagonal matrix consisting of the descending unbounded spiked eigenvalues, and $\bbD_{2}$ is the diagonal matrix of non-spiked eigenvalues. For the first time we establish CLT for LSS of sample covariance matrix with a more general kernel function $ f $  without the assumption that the spectrums of population covariance matrix are bounded. The result in this work clearly overcomes the drawback of spectrum bounded assumption in existing literatrue. Another important improvement of our work over majority of known results in the literature is that our results hold when kernel function $ f $  are not just polynomial functions. The most commonly used logarithmic function is allowed in this work. Furthermore, the number of divergent eigenvalues in our work is allowed to increase to infinity. Our results offer theoretical guarantee for obtaining the asymptotic distribution for LSS of sample covariance matrix, which can make the application of LSS more extensive instead of restricting with spectrum bounded condition.
One salient feature of our result is that the asymptotic distribution of LSS is varied depending on the divergence order of $ \phi_{n}\left(\al_{j} \right)f'\left( \phi_{n}\left(\al_{j} \right) \right)  $, which is consistent with \cite{yin2021spectral}. 

Based on previous knowledge, it seems that when the number of spiked eigenvalues is small, the influence caused by spiked part is little and the distribution of LSS is mainly decided by the bulk part, however, that is not the case. After sumlations, a surprise is that when the spiked eigenvalues is very large, LSS will also affected by the spiked part even though the number of spike is small. Some useful results we refer to \cite{bai2010spectral}, which is comprehensive and indepth.

The remaining sections are organized as follows. Section 2 gives a detailed description of notations and assumptions. Main result on the CLT for LSS of sample covariance matrix when the spectrum is unbounded is stated in Section 3. In Section 4, we present the results of our sumlation studies.  Lemmas and technical proofs of the theorem  are relegated to Section 5.

\section{Notations and assumptions}

Throughout the paper, we use bold Roman capital letters to represent matrices,e.g., $ \bSi $. Scalars are often in lowercase letters. Vectors follow bold italic style like $ \boldsymbol x_{i} $. We use tr$ (\bbA) $  to denote the trace of matrix $ \bbA $, $ \bbA' $ and $ \bbA^{\ast} $ denote the transpose and conjugate transpose of matrix $ \bbA $ respectively. Let $ \left[\bbA \right]_{ij}  $ denote the $ (i,j) $-th entry of the matrix $ \bbA $ and $ \oint_{\mathcal{C}}f(z)dz $ denote the contour integral of $ f(z) $ on the contour $ \mathcal{C} $.  A sequence of random vector $ \boldsymbol x_{n} $ weak converging to a random vector $ \boldsymbol x $ is denoted by $  \boldsymbol x_{n}\stackrel{d}{\rightarrow} \boldsymbol x$.

Define $ \bbX=(\boldsymbol x_{1},\ldots,\boldsymbol x_{n})=(x_{ij}) $,$ 1\leq i\leq p $,$ 1\leq j\leq n $, and it consists of independent random variables with mean 0 and variance 1. Furthermore, $E{x_{ij}^2}=0 $ for the complex case. Let $ \bbTp $ is a $p \times p$ deterministic matrix, then $\bbTp\bbX$ can be seen as a random sample from the population with the population covariance matrix $\bbTp\bbTp^{\ast}=\bSi$. Denote $\al_{1},\ldots,\al_{K}$ with the multiplicity $ m_{k}, k=1,\ldots,K $ as the spiked eigenvalues of  $ \bSi $, where $ m_{1}+\cdots+m_{K}=M$, and $ \dfrac{M}{n}=o\left( 1 \right) $. The corresponding sample covariance matrix of the observations $\bbTp\bbX$ is defined as $$\bbB=\dfrac{1}{n}\bbTp\bbX\bbX^{\ast}\bbTp^{\ast}$$
then $\bbB$ is the so-called generalized spiked sample covariance matrix. Denote $\lambda_{1},\lambda_{2},\ldots,\lambda_{M},\lambda_{M+1},\ldots,\lambda_{p}$ as the eigenvalues of the  $\bbB$. The above model includes the spiked population model as a special case, and the spiked population model was orginally proposed by \cite{10.1214/aos/1009210544}. Other extensions on this topic can be found in \cite{BAIK20061382, 10.2307/24307692, 10.1214/07-AIHP118, BAI2012167, 10.3150/20-BEJ1237, pg4mt,cai2020limiting, li2020asymptotic, zhang2020asymptotic}, and so on.
Define the singular value decomposition of $\bbTp$ as 
\begin{equation*}
	\bbTp=
	\bbV
	\left( 
	\begin{array}{cc}
		\bbD_{1}^\frac{1}{2} & 0\\
		0 & \bbD_{2}^\frac{1}{2}
	\end{array}
	\right)
	\bbU^{\ast}
\end{equation*}
where $\bbU$ and $\bbV$ are unitary matrices. $\bbD_{1}$ is a diagonal matrix of the $ M $ spiked eigenvalues with the components tend to infinity and  $\bbD_{2}$ is the diagional matrix of the non-spiked eigenvalues with the bounded components. 

The basic limit theorem on the eigenvalues of $ \bbB $ concerns its empirical spectral distribution $ F^{\bbB} $, where for any matrix $\bbA $, $F^{\bbA}$ denotes the empirical spectral distribution of $\bbA $, that is, if $\bbA $ is  $ p\times p $ then
\begin{equation*}
	F^{\bbA}\left(x \right)=\frac{1}{p}\left(\text{number of eigenvalues of }  \bbA \leq x \right)  	
\end{equation*}	
If for all $ i,j $, $ x_{ij} $ are independent elements, and with probablity 1, $F^{\bSi} $ converges weakly to a proper cumulative distribution function (c.d.f.) $ H $ under the condition that  $ \dfrac{p}{n}\rightarrow c>0 $ as $ p,n\rightarrow \infty $, then with probablity 1 $ F^{\bbB} $ converges in distribution to $ F^{c,H} $, a nonrandom proper c.d.f.
From definition we know $ F^{\bbB}=\frac{1}{p}\sum_{i=1}^{p}\mathbbm{1}_{\left\lbrace \lambda_{i}\leq x\right\rbrace } $, where $ \mathbbm{1}_{\left\lbrace \cdot\right\rbrace}  $ denotes the indicator function. We call $\sum_{j=1}^{p}f\left( \bgl_{j}\right) $ linear spectral statistics of the sample covariance matrix, where kernel function $ f $ is analytic. Note that,  theoretically, the condition that `f is analytic' can be relaxed to `f has fourth-order or fifth-order continous derivative' (see \cite{bai2010functional}). Practically, most of functions we use are analytic functions, such as logarithmic function, so it is acceptable to assume `f is analytic'. 

To derive the CLT for LSS of generalized sample covariance matrix, the following assumptions are made:
\newtheorem{assumption}{Assumption}[]
\begin{assumption} 
$ x_{ij} $ \textit{are independent random variables with common moments} $$ Ex_{ij}=0,\quad  E\left| x_{ij}\right| ^{2}=1, \quad  \beta_{x}= E\left|x_{ij}\right| ^{4}- \left|Ex_{ij}^{2}\right|^{2}-2,  \quad \alpha_{x}=\left|Ex_{ij}^{2}\right|^{2}  $$ 
\textit { and satisfy the following Lindeberg-type condition: }
\begin{align*}
	&\frac{1}{n p} \sum_{i=1}^{p} \sum_{j=1}^{n} E\left\{\left|x_{ij}\right|^{4} \mathbbm{1}_{\left\{\left|x_{ij}\right| \geq \eta \sqrt{n}\right\}}\right\} \rightarrow 0 \quad \text { for any fixed } \eta>0
\end{align*}	
\end{assumption} 
\begin{assumption} 
\textit{The ratio of dimension and sample size (RDS)} $ c_{n}=\dfrac{p}{n}\rightarrow c>0, $ \textit{as both} $ n $ \textit{and} $ p $ \textit{go to infinity simultaneously}.
\end{assumption} 
\begin{assumption}
$ \bbTp $ \textit{is a} $p \times p$ \textit{deterministic matrix with unbounded spectral norm, and it admits the singular decomposition}
\begin{equation*}
	\bbTp=
	\bbV
	\left( 
	\begin{array}{cc}
		\bbD_{1}^\frac{1}{2} & 0\\
		0 & \bbD_{2}^\frac{1}{2}
	\end{array}
	\right)
	\bbU^{\ast}.
\end{equation*}
\textit{The population covariance matrix is} $ \bSi=\bbTp\bbTp^{\ast} $, \textit{and the spiked eigenvalues of matrix} $ \bSi $, $ \al_{1},\al_{2},...\al_{K} $, \textit{with multiplicities} $ m_{1},...,m_{K}, $ \textit{satisfy} $ \phi'\left(\al_{k} \right)>0  $, \textit{for} $ 1\leq k\leq K $,  \textit{where} $ \phi\left(x \right)=x\left(1+c\int\dfrac{t}{x-t}dH\left( t\right)  \right)   $. \textit{Moreover},  $ m_{1}+\cdots+m_{K}=M,  \dfrac{M}{n}=o\left( 1 \right)  $. 	
\end{assumption} 
\begin{assumption}
$ \bbTp $ \textit{is real or the variables} $x_{ij}$ \textit{are complex satisfying} $\alpha_{x}=0 .$ 
\end{assumption}
\begin{assumption}
$\bbTp^{\ast}\bbTp$ \textit{is diagonal or} $\beta_{x}=0.$	 
\end{assumption}
\newtheorem{remark}[thm]{Remark}
\begin{remark}
\textit{Assumption 4 is for the second-order moment condition of $ x_{ij} $ and Assumption 5 is for the fourth-order moment condition of $ x_{ij} $, they were firstly proposed in \cite{10.1214/14-AOS1292}}, \textit{which mean the Gaussian-like moment conditions can be alternative.}
\end{remark}
\begin{assumption}
\textit{Kernel function $ f_{1},\cdots, f_{h} $ are anlaytic on an open domain of the complex plan containing the support of the limiting spectral distribution (LSD)} $ F^{c,H}, $ \textit{ and satisfy:}  $ x_{n},y_{n} \rightarrow \infty, then \dfrac{f_{l}'\left(x_{n} \right) }{f_{l}'\left(y_{n} \right)}\rightarrow 1 $ \textit{provided} $ \dfrac{x_{n}}{y_{n}}\rightarrow 1, l=1,2,\cdots,h.  $  
\end{assumption}    
\begin{remark}
	\textit{In fact, Assumption 6 is not too restrictive to use, in the real data,  many common functions such as logarithmic function can satisify. Further, the condition in Assumption 6 can be relaxed to $ \dfrac{x_{n}}{y_{n}}-1= O\left(\dfrac{1}{\sqrt{n}} \right). $ }
\end{remark}

\section{Main Results}  
Define $$ G_{n}\left( x\right)=p\left[F^{\bbB}\left(x \right)-F^{c_{n},H_{n}}\left(x \right)   \right],  $$
$$\rho_{l}=\dfrac{1}{\sqrt{\sum_{k=1}^{K}\left(\dfrac{\phi_{n}\left(\al_{k} \right) }{\sqrt{n}}f'_{l}\left(\phi_{n}\left(\al_{k} \right) \right)  \right)^{2}+1 }} \quad l=1,2,\cdots,h, $$ 
$$  \boldsymbol Y_{n}\left(f_{l} \right)=\rho_{l} \left[ \int f_{l}\left(x \right)dG_{n}\left( x\right)-\sum_{k=1}^{K}\sum_{j\in J_{k}}f_{l}\left(\phi_{n}\left(\al_{j} \right)  \right)-\dfrac{M}{2\pi i}\oint_{\mathcal C}f_{l}\left(z \right)\dfrac{\underline{m}'(z)}{\underline{m}(z)}dz\right].  $$
The main result is stated in the following theorem.
\begin{thm}\label{thm1}
\textit{Under Assumption} $ 1 $ \textit{to} $6 $, \textit{the random vector } 
$ \left( \boldsymbol Y_{n}\left(f_{1} \right),...,\boldsymbol Y_{n}\left(f_{h} \right)\right)  $ \textit{ converges weakly to a h-dimensional Gaussian vector} $\left(  \boldsymbol Y_{f_{1}},...,\boldsymbol Y_{f_{h}}\right) $, \textit{with mean function}
\begin{align*}
	E\boldsymbol Y_{f_{l}}&=\rho_{l}\mu_{l},
\end{align*}
\textit{and covariance function}
\begin{align*}
	Cov\left(\boldsymbol Y_{f_{s}}, \boldsymbol Y_{f_{t}}\right)
	&=\dfrac{1}{n}\rho_{s}\rho_{t}\sum_{k=1}^{K}\phi_{n}^{2}\left(\al_{k} \right)f'_{s}\left(\phi_{n}\left(\al_{k} \right)  \right)f'_{t}\left(\phi_{n}\left(\al_{k} \right)  \right)\sigma_{k}^{2}+\rho_{s}\rho_{t}\sigma^{2}_{s,t}
\end{align*} 
\textit{where}
\begin{align*}
\mu_{l}=&-\frac{\alpha_{x}}{2 \pi i}\cdot\oint_{\mathcal{C}}f_{l}(z)\dfrac{  c \int \underline{m}^{3}(z)t^{2}\left(1+t \underline{m}(z)\right)^{-3} d H(t)}{\left(1-c \int \frac{\underline{m}^{2}(z) t^{2}}{\left(1+t \underline{m}(z)\right)^{2}} d H(t)\right)\left(1-\alpha_{x} c \int \frac{\underline{m}^{2}(z) t^{2}}{\left(1+t \underline{m}(z)\right)^{2}} d H(t)\right) }dz \\
&-\frac{\beta_{x}}{2 \pi i} \cdot \oint_{\mathcal{C}} f_{l}(z) \frac{c \int \underline{m}^{3}(z) t^{2}\left(1+t \underline{m}(z)\right)^{-3} d H(t)}{1-c \int \underline{m}^{2}(z) t^{2}\left(1+t \underline{m}(z)\right)^{-2} d H(t)} dz,\\	
\sigma_{k}^{2}=&\dfrac{\sum_{j\in J_{k}}\left( \left(q+1 \right)\theta_{k}+\pi_{x,jjjj}\nu_{k}\right) +\sum_{j_{1}\neq j_{1}}\pi_{x,j_{1}j_{1}j_{2}j_{2}}\nu_{k} }{\theta_{k}^{2}},\\
\sigma^{2}_{s,t}=&-\frac{1}{4 \pi^{2}} \oint_{\mathcal{C}_{1}} \oint_{\mathcal{C}_{2}} \frac{f_{s} \left(z_{1}\right) f_{t}\left(z_{2}\right)}{\left(\underline{m} {\left.\left(z_{1}\right)-\underline{m}\left(z_{2}\right)\right)^{2}}\right.} d \underline{m}\left(z_{1}\right) d \underline{m}\left(z_{2}\right) 
-\frac{c \beta_{x}}{4 \pi^{2}} \oint_{\mathcal{C}_{1}} \oint_{\mathcal{C}_{2}} f_{s}\left(z_{1}\right) f_{t}\left(z_{2}\right)\\
&\left[\int \frac{t}{\left(\underline{m}\left(z_{1}\right) t+1\right)^{2}}\right.
\left.\times \frac{t}{\left(\underline{m}\left(z_{2}\right) t+1\right)^{2}} d H(t)\right] d \underline{m}\left(z_{1}\right) d \underline{m}\left(z_{2}\right)\\
&\quad-\frac{1}{4 \pi^{2}} \oint_{\mathcal{C}_{1}} \oint_{\mathcal{C}_{2}} f_{s}\left(z_{1}\right) f_{t}\left(z_{2}\right)\left[\frac{\partial^{2}}{\partial z_{1} \partial z_{2}} \log \left(1-a\left(z_{1}, z_{2}\right)\right)\right] d z_{1} d z_{2}. \end{align*}
\textit{In the equation above,}  $ \underline{m}(z) $ \textit{is the stieltjes transform of} $ \underline{F}^{c,H} $, $$ \pi_{x,i_{1}j_{1}i_{2}j_{2}}=\mathrm{lim}\sum_{t=1}^{p}\overline{u}_{ti_{1}}u_{tj_{1}}u_{ti_{2}}\overline{u}_{tj_{2}}\\E\left\lbrace\left|x_{11} \right|^{4} \mathbbm{1}\left\lbrace\left|x_{11} \right|<\sqrt{n}  \right\rbrace -2-q   \right\rbrace, 1\leq j\leq M,  $$ $$\theta_{k}=\phi_{k}^{2}\underline{m}_{2}\left( \phi_{k}\right), \nu_{k}=\phi_{k}^{2} \underline{m}^{2}\left(\phi_{k}\right)$$$$ \underline{m}_{2}\left( \lambda\right)=\int\dfrac{1}{\left( \lambda-x\right) ^{2}}d\underline{F}\left( x\right), \underline{m}\left( \lambda\right)=\int\dfrac{1}{x- \lambda }d\underline{F}\left( x\right). $$ \textit{with} $ \underline{F} $ \textit{being the LSD of matrix} $ n^{-1}\bbX^{\ast}\bGma\bbX$, 
$ \bbU_{1}=\left(u_{ij} \right)_{i=1,2,...p;j=1,2,...M}  $,  and $ q=1 $ \textit{for the real case},  $ q=0 $  \textit{for the complex case} , $ \phi_{k}=\phi\left(x \right)\mid_{x=\al_{k}}=\al_{k}\left(1+c\int\dfrac{t}{\al_{k}-t}dH\left(t \right)  \right)  $, $ \phi_{n}\left( \al_{k}\right) $ \textit{has the same form as} $ \phi_{k} $ \textit{but with} $ c_{n},H_{n} $ \textit{instead of} $ c,H $. $\mathcal{C}, \mathcal{C}_{1}$ \textit{and} $\mathcal{C}_{2}$ \textit{are closed contours in the complex plan enclosing the support of the limiting spectral distribution (LSD)} $ F^{c, H}$, \textit{and} $\mathcal{C}_{1}$ \textit{and} $\mathcal{C}_{2}$ \textit{being nonoverlapping, the function} $a\left(z_{1}, z_{2}\right)$ \textit{is}
$$
a\left(z_{1}, z_{2}\right)=\alpha_{x}\left(1+\frac{\underline{m}\left(z_{1}\right) \underline{m}\left(z_{2}\right)\left(z_{1}-z_{2}\right)}{\underline{m}\left(z_{2}\right)-\underline{m}\left(z_{1}\right)}\right).
$$

\end{thm}
\begin{remark}
\textit{From CLT above it is not difficult to see the asymptotic distribution is varied depending on the divergence order of} $ \rho_{l} $, \textit{which is also the divergence order of} $ \dfrac{ \phi_{n}\left(\al_{k} \right)}{\sqrt{n}} f'_{l}\left(\phi_{n}\left(\al_{k} \right) \right)  $. \textit{It means the approximate divergence order of the spiked eigenvalues decides the asymptotic results}, \textit{which is consistent with} \cite{yin2021spectral}.  
\end{remark}
\begin{remark}
\textit{In the work} \cite{yin2021spectral}, \textit{they obtain the asymptotic CLT when kernel functions are polynomial functions. Comparing with it, we allow more general functions such as logarithmic functions.} 	
\end{remark}
\begin{remark}
	In particular, when $ \dfrac{ \phi_{n}\left(\al_{1} \right)}{\sqrt{n}} f'_{l}\left(\phi_{n}\left(\al_{1} \right) \right)\rightarrow 0 \left( n\rightarrow \infty \right),$ Theorem \ref{thm1} is inline with Theorem 2.1 in \cite{10.1214/14-AOS1292}, and when $ \dfrac{ \phi_{n}\left(\al_{1} \right)}{\sqrt{n}} f'_{l}\left(\phi_{n}\left(\al_{1} \right) \right)\rightarrow \infty \left( n\rightarrow \infty \right),$ Theorem \ref{thm1} is in accord with the CLT derived by \cite{10.3150/20-BEJ1237}. Moreover, $ \dfrac{\phi_{n}\left(\al_{k} \right)}{\sqrt{n}}f_{l}'\left( \phi_{n}\left(\al_{k} \right)\right) $ \textit{is obtained from the spiked part of the linear spectrum statistics}  $ \sum_{j=1}^{M}f_{l}\left(\lambda_{j} \right)-\sum_{k=1}^{K}\sum_{j\in J_{k}}f_{l}\left(\phi_{n}\left(\al_{j} \right)  \right)  $ \textit{(which is shown in the Section 5)}.
\end{remark}

\section{Simulation}
We conducted a number of simulation studies to examine the asymptotic distribution of linear spectral statistics when the spectrums of $ \bSi $ are unbounded. In our experiments, we define $ \bbZ = \bSi^{\frac{1}{2}}\bbX $, where $ \bbX=(\boldsymbol x_{1},...,\boldsymbol x_{n})  $, and $ \left\lbrace \boldsymbol x_{i}  \right\rbrace_{1\leq i\leq n}  $ are $ p $-variate random vectors, and we assume that they come from Gaussian distribution. Clearly, sample covariance matrix $ \bbB=\dfrac{1}{n}\bbZ\bbZ' $.  Moreover, we set $ p=100 $, $ n=3000 $. About simulation we refer to \cite{zheng2017clt,jonsson1982some}. For the setting of eigenvalues of  $ \bSi $, we consider the following spectrums: 
\begin{enumerate}[(1)]  
\item  $ \operatorname{Spec}\left(\bSi\right)=\{\underbrace{n^{1/3}, \cdots, n^{1/3}}_{6}, \underbrace{n^{1/3}-1, \cdots, n^{1/3}-1}_{6}, \underbrace{n^{1/4}-2, \cdots, n^{1/4}-2}_{6}, \underbrace{1, \cdots, 1}_{p-18}\} ; $ 
\item $ \operatorname{Spec}\left(\bSi\right)=\{\underbrace{n^{1/2}, \cdots, n^{1/2}}_{6}, \underbrace{n^{1/3}-1, \cdots, n^{1/3}-1}_{6}, \underbrace{n^{1/4}-2, \cdots,  n^{1/4}-2}_{6}, \underbrace{1, \cdots, 1}_{p-18}\} ; $   
\item $ \operatorname{Spec}\left(\bSi\right)=\{\underbrace{n, \cdots, n}_{6}, \underbrace{n^{1/2}, \cdots, n^{1/2}}_{6}, \underbrace{n^{1/3}, \cdots, n^{1/3}}_{6}, \underbrace{1, \cdots, 1}_{p-18}\} ;  $
\end{enumerate}

We highlight some salient features of our choices of the spiked eigenvalues of $ \bSi $: Similarly, all the spectrums above satisify that the differences between two adjacent different spiked eigenvalues are divergent after multiplying by $\sqrt{n} $, so as to ensure the limiting distribution of spiked eigenvalues nonoverlaped. Moreover, the divergence speed of the spectrum norm of (1) is smaller than $ \sqrt{n} $, while 
the divergence speed of the that of (2) and (3) is equal to $ \sqrt{n} $ and larger than $ \sqrt{n} $.

In our framework,  we study the cases when kernel function $ f(x)=x $, and we suppose $ \bbx_{ij} $ and $ \bbTp $ are real. The number of spike is: $ M=18 $.
In Figure $ 1-3 $, we compare the empirical density (the blue histogram) of  the linear spectral statistics when the spectrums of population covariance matrix are $ (1)-(3) $ with the standard
normal density curve (the red line) with 3000 repetitions.

\begin{figure}[H]
	\centering
	\begin{minipage}[t]{0.3\textwidth}
		\centering
		\includegraphics[width=4.5cm]{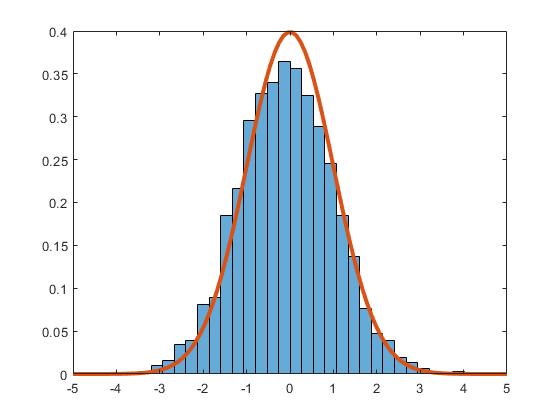}
		\captionsetup{font={scriptsize}}
		\caption{(p, n) = (100, 3000) under spectrum (1), M=18}
	\end{minipage}
	\begin{minipage}[t]{0.3\textwidth}
		\centering
		\includegraphics[width=4.5cm]{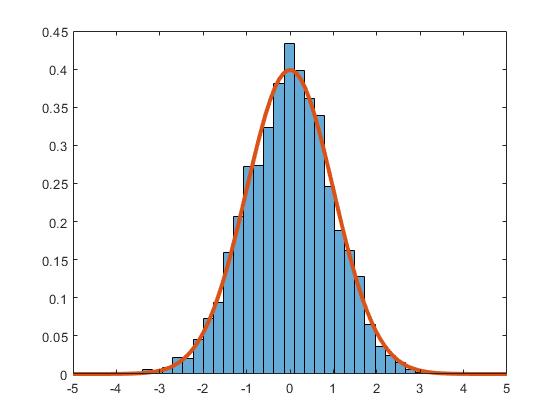}
		\captionsetup{font={scriptsize}}
		\caption{(p, n) = (100, 3000) under spectrum (2), M=18}
	\end{minipage}
  \begin{minipage}[t]{0.3\textwidth}
	    \centering
	    \includegraphics[width=4.5cm]{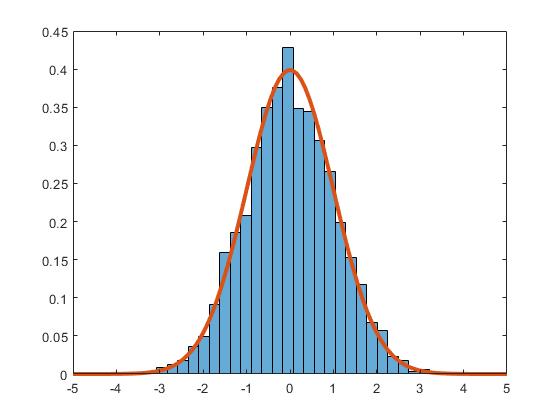}
	    \captionsetup{font={scriptsize}}
	    \caption{(p, n) = (100, 3000) under spectrum (3), M=18}
  \end{minipage}
\end{figure}

\section{Technical proofs}
\subsection{Some primary definitions and lemmas}
In this section, we will give some useful results that will be later  used in the proof of Theorem \ref{thm1}. For the population covariance matrix $\bSi=\bbTp\bbTp^{\ast}$, consider the corresponding sample covariance matrix $ \bbB=\bbTp\bbS_{x}\bbTp^{\ast} $, where $ \bbS_{x}=\frac{1}{n}\bbX\bbX^{\ast}  $
By singular value decomposition of $ \bbTp $,
\begin{equation*}
	\bbB=\bbV\left(\begin{array}{cc}
		\bbD_{1}^\frac{1}{2}\bbU_{1}^{\ast}\bbS_{x}\bbU_{1}\bbD_{1}^\frac{1}{2}  &\bbD_{1}^\frac{1}{2}\bbU_{1}^{\ast}\bbS_{x}\bbU_{2}\bbD_{2}^\frac{1}{2}\\
		\bbD_{2}^\frac{1}{2}\bbU_{2}^{\ast}\bbS_{x}\bbU_{1}\bbD_{1}^\frac{1}{2}	&\bbD_{2}^\frac{1}{2}\bbU_{2}^{\ast}\bbS_{x}\bbU_{2}\bbD_{2}^\frac{1}{2}
	\end{array} \right)\bbV^{\ast}
\end{equation*}
where $ \bbU=\left(\bbU_{1},\bbU_{2} \right)  $.
Denote 
\begin{equation*}
	\bbS=\left(\begin{array}{cc}
		\bbD_{1}^\frac{1}{2}\bbU_{1}^{\ast}\bbS_{x}\bbU_{1}\bbD_{1}^\frac{1}{2}  &\bbD_{1}^\frac{1}{2}\bbU_{1}^{\ast}\bbS_{x}\bbU_{2}\bbD_{2}^\frac{1}{2}\\
		\bbD_{2}^\frac{1}{2}\bbU_{2}^{\ast}\bbS_{x}\bbU_{1}\bbD_{1}^\frac{1}{2}	&\bbD_{2}^\frac{1}{2}\bbU_{2}^{\ast}\bbS_{x}\bbU_{2}\bbD_{2}^\frac{1}{2}
	\end{array} \right)	\triangleq\left(\begin{array}{cc}\bbS_{11}&\bbS_{12}\\\bbS_{21}&\bbS_{22}\end{array} \right) 
\end{equation*}
Note that $ \bbB $ and $ \bbS $ have the same eigenvalues. 

The analysis below is carried out by dividing the linear spectral statistics $ \sum_{j=1}^{p}f\left( \bgl_{j}\right) $ into two parts: the spiked part $ \sum_{j=1}^{M}f\left( \bgl_{j}\right) $ and the bulk part $ \sum_{j=M+1}^{p}f\left( \bgl_{j}\right) $. From \cite{10.1214/aop/1078415845} or \cite{10.1214/14-AOS1292} we know the center part of $ \sum_{j=1}^{p}f\left( \bgl_{j}\right) $ is $ p\int f\left( x\right)dF^{c_{n},H_{n}}\left( x\right) $, where $ c_{n}=\frac{p}{n} $, $ H_{n}=F^{\bGma} $, note that here $ \bGma=\bbU_{2}\bbD_{2}\bbU_{2}^{\ast} $. By Cauchy integral formula, $p\int f\left( x\right)dF^{c_{n},H_{n}}\left( x\right)=-\frac{1}{2\pi i}\oint_{C}f\left(z \right)pm_{1n0}\left(z \right)dz $, where $ m_{1n0} $ is the stietjes transform of $ F^{c_{n},H_{n}} $.  

Consider the matrix $ \bbS_{22} $, denote $ \widetilde{\bgl_{j}} $ the eigenvalues of $ \bbS_{22} $, so the linear spectral statistics of $ \bbS_{22} $ is $ \sum_{j=1}^{p-M}f\left( \widetilde{\bgl_{j}}\right) $. $ \left( p-M\right)\int f\left(x \right)dF^{c_{nM},H_{2n}} $ is the corresponding center part of $ \sum_{j=1}^{p-M}f\left( \widetilde{\bgl_{j}}\right) $, where $ c_{nM}=\frac{p-M}{n} $, $ H_{2n}=F^{\bbD_{2}} $. By Cauchy integral formula, $\left( p-M\right)\int f\left(x \right)dF^{c_{nM},H_{2n}}=-\frac{1}{2\pi i}\oint_{C}f\left(z \right)\left( p-M\right) m_{2n0}\left(z \right)dz.$ $ m_{2n0} $ is the stietjes transform of $ F^{c_{nM},H_{2n}} $. 

We denote  $\underline{\bbB}=\dfrac{1}{n}\bbX^{\ast}\bbTp^{\ast}\bbTp\bbX $, (the spectral of which differs from that of $ \bbB $ by $ \left|  n-p\right|  $ zeros). Its limiting spectral distribution is $ \underline{F}^{c,H} $, $\underline{F}^{c,H}\equiv\left(1-c \right)\mathbbm{1}_{\left[0,\infty \right) }+cF^{c,H}   $, and its stieltjes transform is $ \underline{m}\left(z \right)  $ .  

\begin{lemma}\label{lemma1}
\textit{The difference between two center part} $ \left( p-M\right)\int f\left(x \right)dF^{c_{nM},H_{2n}} $ \textit{and} $p\int f\left( x\right)dF^{c_{n},H_{n}}\left( x\right) $ \textit{is} $ 0 $,
\textit{where} $ f $ \textit{is an anlaytic function}. 
\end{lemma}
\begin{proof}
By Cauchy integral formula, 
$$ p\int f(x)dF^{c_{n},H_{n}}=-\dfrac{1}{2 \pi i}\oint_\mathcal{C}f(z)pm_{1n0}dz=-\dfrac{n}{2 \pi i}\oint_\mathcal{C}f(z)\underline{m}_{1n0}dz, $$ $$ (p-M)\int f(x)dF^{c_{nM},H_{2n}}=-\dfrac{1}{2 \pi i}\oint_\mathcal{C}f(z)pm_{2n0}dz=-\dfrac{n}{2 \pi i}\oint_\mathcal{C}f(z)\underline{m}_{2n0}dz, $$
where $m_{1n0} $ and $ m_{2n0}$ are respectively the stieltjes transform of $F^{c_{n},H_{n}} $ and $ F^{c_{nM},H_{2n}}$.
So $ (p-M)\int f(x)dF^{c_{nM},H_{2n}}-p\int f(x)dF^{c_{n},H_{n}}=\dfrac{n}{2 \pi i}\oint_\mathcal{C}f(z)\left(m_{1n0}-m_{2n0}\right) dz. $

Note that
$ m_{2n0} $ and $ m_{1n0} $ are respectively the unique solution to 
\begin{align}
	z=-\dfrac{1}{\underline{m}_{1n0}}+c_{n}\int\dfrac{tdH_{n}\left( t\right) }{1+t\underline{m}_{1n0}}\label{1}\\     
	z=-\dfrac{1}{\underline{m}_{2n0}}+c_{nM}\int\dfrac{tdH_{2n}\left( t\right) }{1+t\underline{m}_{2n0}} \label{2} 
\end{align}
where $ \underline{m}_{1n0}=-\dfrac{1-c_{n}}{z}+c_{n}m_{1n0} $, $ \underline{m}_{2n0}=-\dfrac{1-c_{nM}}{z}+c_{nM}m_{2n0} $. $ c_{n}=\dfrac{p}{n} $, $ c_{nM}=\dfrac{p-M}{n} $, $ H_{n}=F^{\bbU_{2}\bbD_{2}\bbU_{2}^{\ast}} $, $ H_{2n}=F^{\bbD_{2}}. $\\
Thus $  H_{n}=\frac{1}{p}\left[\sum_{i=1}^{M}\mathbbm{1}_{\left\lbrace 0\leq t\right\rbrace }+\sum_{i=M+1}^{p}\mathbbm{1}_{\left\lbrace \alpha_{i}\leq t\right\rbrace } \right]=\dfrac{M}{p}+\dfrac{1}{p}\sum_{i=M+1}^{p}\mathbbm{1}_{\left\lbrace \alpha_{i}\leq t\right\rbrace }  $, $ H_{2n}=\dfrac{1}{p-M}\sum_{i=M+1}^{p}\mathbbm{1}_{\left\lbrace\alpha_{i}\leq t \right\rbrace } $.\\
So (\ref{1}) can be written into 
\begin{align}
	\nonumber z&=-\dfrac{1}{\underline{m}_{1n0}}+\dfrac{p}{n}\int\dfrac{td\left( \dfrac{M}{p}+\dfrac{1}{p}\sum_{i=M+1}^{p}\mathbbm{1}_{\left\lbrace \alpha_{i}\leq t\right\rbrace }\left( t\right)\right)  }{1+t\underline{m}_{1n0}}\\	
	\nonumber&=-\dfrac{1}{\underline{m}_{1n0}}+\dfrac{p}{n}\int\dfrac{td\left( \dfrac{1}{p}\sum_{i=M+1}^{p}\mathbbm{1}_{\left\lbrace \alpha_{i}\leq t\right\rbrace }\left( t\right)\right)  }{1+t\underline{m}_{1n0}}\\
	&=-\dfrac{1}{\underline{m}_{1n0}}+\dfrac{1}{n}\sum_{i=M+1}^{p}\dfrac{\alpha_{i}}{1+\al_{i}\underline{m}_{1n0}}.\label{3}
\end{align}	
Similarly, the equation (\ref{2}) can be written into 
\begin{align}
	z=-\dfrac{1}{\underline{m}_{2n0}}+\dfrac{1}{n}\sum_{i=M+1}^{p}\dfrac{\al_{i}}{1+\al_{i}\underline{m}_{2n0}}.\label{4}
\end{align}	
So, from  (\ref{3})-(\ref{4})  we have 
\begin{align*}
	0&=\dfrac{1}{\underline{m}_{2n0}}-\dfrac{1}{\underline{m}_{1n0}}+\dfrac{1}{n}\sum_{i=M+1}^{p}\left( \dfrac{\al_{i}}{1+\al_{i}\underline{m}_{1n0}}-\dfrac{\al_{i}}{1+\al_{i}\underline{m}_{2n0}}\right) \\
	&=\dfrac{\underline{m}_{1n0}-\underline{m}_{2n0}}{\underline{m}_{1n0}\underline{m}_{2n0}}+\dfrac{1}{n}\sum_{i=M+1}^{p}\dfrac{\al_{i}^{2}\left( \underline{m}_{2n0}-\underline{m}_{1n0}\right) }{\left( 1+\al_{i}\underline{m}_{1n0}\right)\left( 1+\al_{i}\underline{m}_{2n0}\right)  }\\
	&=\dfrac{\underline{m}_{1n0}-\underline{m}_{2n0}}{\underline{m}_{1n0}\underline{m}_{2n0}}-\dfrac{1}{n}\sum_{i=M+1}^{p}\dfrac{\al_{i}^{2}\left( \underline{m}_{1n0}-\underline{m}_{2n0}\right) }{\left( 1+\al_{i}\underline{m}_{1n0}\right)\left( 1+\al_{i}\underline{m}_{2n0}\right)  }\\
\end{align*}	
Thus 
\begin{align*}
	\left(\underline{m}_{1n0}-\underline{m}_{2n0} \right)\left(\dfrac{1}{\underline{m}_{1n0}\underline{m}_{2n0}}-\dfrac{1}{n}\sum_{i=M+1}^{p}\dfrac{\al_{i}^{2}}{\left( 1+\al_{i}\underline{m}_{2n0}\right)\left( 1+\al_{i}\underline{m}_{2n0}\right)  } \right)  =0. 
\end{align*}
So $ \underline{m}_{1n0}-\underline{m}_{2n0}=0 $. Accordingly, $$(p-M)\int f(x)dF^{c_{nM},H_{2n}}-p\int f(x)dF^{c_{n},H_{n}}=\dfrac{n}{2 \pi i}\oint_\mathcal{C}f(z)\left(m_{1n0}-m_{2n0}\right) dz=0.$$ 
The proof is finished.
\end{proof}

Note that the non-spiked part $ \sum_{j=M+1}^{p}f\left(\lambda_{j} \right)  $ cannot use the conclusion of \cite{10.1214/aop/1078415845} or \cite{10.1214/14-AOS1292} directly, because the off-diagonal block $ \bbS_{12} $ and $ \bbS_{21} $ are not 0, while the linear spectral statistics generated by $ \bbS_{22} $ can use the conclusion of \cite{10.1214/aop/1078415845} or \cite{10.1214/14-AOS1292} directly, so in the following lemma we calculate the difference between  $ \sum_{j=M+1}^{p}f\left( \lambda_{j}\right)$ and $\sum_{j=1}^{p-M}f\left(\widetilde{\bgl_{j}} \right) $.\\

\begin{lemma}\label{lemma2}
\textit{The difference between}  $\sum_{j=M+1}^{p}f\left( \lambda_{j}\right) $  \textit{and}  $ \sum_{j=1}^{p-M}f\left( \widetilde{\bgl_{j}}\right)$ \textit{tends to}  $ \frac{M}{2\pi i}\oint_\mathcal{C}f\left(z \right)\dfrac{\underline{m}^{'}(z)}{\underline{m}(z)}dz$, \textit{where} $ f $ \textit{is an anlaytic function}. $ \underline{m}\left(z \right)  $ \textit{is the stieltjes transform of} $ \underline{F}^{c,H} $, \textit{which is the limiting spectral distribution of} $ \underline{\bbB} $. \textit{The coutour} $ \mathcal{C} $ \textit{encloses the support of} ${F}^{c,H} $.	
\end{lemma}
\begin{proof}
We denote
\begin{align*}
	L_{1}\triangleq\sum_{j=M+1}^{p}f\left( \lambda_{j}\right).
\end{align*}
By Cauchy integral formula, we have
\begin{align*}
	L_{1}=-\frac{p}{2\pi i}\oint_{\mathcal C}f\left(z \right)m_{n}\left(z \right)dz	
\end{align*}
where $m_{n}=\frac{1}{p}\mathrm{tr}\left(\bbS-z\bbI_{p} \right)^{-1}=\frac{1}{p}\mathrm{tr}\left(\bbB-z\bbI_{p} \right)^{-1} $. Here the contour $ \mathcal{C} $ is closed, enclosing the support of ${F}^{c,H} $.
Similarly,
\begin{align*}
	L_{2}\triangleq\sum_{j=1}^{p-M}f\left( \widetilde{\bgl_{j}}\right)=-\frac{p-M}{2\pi i}\oint_{\mathcal C}f\left(z \right)m_{2n}\left(z \right)dz.	
\end{align*}
where $m_{2n}=\frac{1}{p-M}\mathrm{tr}\left(\bbS_{22}-z\bbI_{p} \right)^{-1} $. Thus, by using block matrix inversion method, we have 
\begin{align*}
	L_{1}-L_{2}=-\frac{1}{2\pi i}\oint_{\mathcal C}f\left(z \right)\left(T_{1}-T_{2} \right) dz.
\end{align*}
where 
\begin{align*}
	T_{1} &=\mathrm{tr}\left( \bbS_{11}-z\bbI_{M}-\bbS_{12}\left( \bbS_{22}-z\bbI_{p-M}\right)^{-1}\bbS_{21} \right)^{-1}\\ 	
	T_{2} &=-\mathrm{tr}\left(\bbS_{22}-zI_{p-M}\right) ^{-1}\bbS_{21}\left( \bbS_{11}-z\bbI_{M}-\bbS_{12}\left(\bbS_{22}-z\bbI_{p-M} \right)^{-1}\bbS_{21} \right)^{-1}\bbS_{12}\left(\bbS_{22}-zI_{p-M} \right)^{-1}   
\end{align*}
From some simple calculations, by using in-out exchange formula
\begin{align*}
	\bbZ\left(\bbZ^{\ast}\bbZ-\lambda \bbI \right)^{-1}\bbZ^{\ast}=\bbI+\lambda\left(\bbZ\bbZ^{\ast}-\lambda \bbI \right)^{-1}  	
\end{align*}	
we have 
\begin{align*}
	T_{1}&=-z^{-1}\mathrm{tr}\left(\bbI_{M}+\frac{1}{n}\bbD_{1}^{\frac{1}{2}}\bbU_{1}^{\ast}\bbX\left(\frac{1}{n}\bbX^{\ast}\bGma\bbX-z\bbI_{n}\right) ^{-1}  \bbX^{\ast}\bbU_{1}\bbD_{1}^{\frac{1}{2}} \right)^{-1}\stackrel{d}\longrightarrow-z^{-1}\mathrm{tr}\left(\bbD^{-1}\underline{m}^{-1}\left( z\right)  \right) 
	\stackrel{d}\longrightarrow0\\  
	T_{2} &=z^{-1}\mathrm{tr}\left(\bbI_{M}+\frac{1}{n}\bbD_{1}^{\frac{1}{2}}\bbU_{1}^{\ast}\bbX\left(\frac{1}{n}\bbX^{\ast}\bGma\bbX-z\bbI_{n}\right) ^{-1}  \bbX^{\ast}\bbU_{1}\bbD_{1}^{\frac{1}{2}} \right)^{-1}\dfrac{1}{n}\bbD_{1}^{\frac{1}{2}}\bbU_{1}^{\ast}\bbX\bbX^{\ast}\bbU_{2}\bbD_{2}^{\frac{1}{2}}\left(\dfrac{1}{n}\bbD_{2}^{\frac{1}{2}}\bbU_{2}^{\ast}\bbX\bbX^{\ast}\bbU_{2}\bbD_{2}^{\frac{1}{2}}-z\bbI_{p-M} \right)^{-2}\\
	&\dfrac{1}{n}\bbD_{2}^{\frac{1}{2}}\bbU_{2}^{\ast}\bbX\bbX^{\ast}\bbU_{1}\bbD_{1}^{\frac{1}{2}}\\
	&=z^{-1}\mathrm{tr}\bbD_{1}^{\frac{1}{2}}\left(\bbI_{M}+\frac{1}{n}\bbD_{1}^{\frac{1}{2}}\bbU_{1}^{\ast}\bbX\left(\frac{1}{n}\bbX^{\ast}\bGma\bbX-z\bbI_{n}\right) ^{-1}  \bbX^{\ast}\bbU_{1}\bbD_{1}^{\frac{1}{2}} \right)^{-1}\dfrac{1}{n}\bbD_{1}^{\frac{1}{2}}\bbU_{1}^{\ast}\bbX\bbX^{\ast}\bbU_{2}\bbD_{2}^{\frac{1}{2}}\left(\dfrac{1}{n}\bbD_{2}^{\frac{1}{2}}\bbU_{2}^{\ast}\bbX\bbX^{\ast}\bbU_{2}\bbD_{2}^{\frac{1}{2}}-z\bbI_{p-M} \right)^{-2}\\
	&\dfrac{1}{n}\bbD_{2}^{\frac{1}{2}}\bbU_{2}^{\ast}\bbX\bbX^{\ast}\bbU_{1}\\
	&=z^{-1}\mathrm{tr}\left(\bbD_{1}^{-1}+\frac{1}{n}\bbU_{1}^{\ast}\bbX\left(\frac{1}{n}\bbX^{\ast}\bGma\bbX-z\bbI_{n}\right) ^{-1}  \bbX^{\ast}\bbU_{1} \right)^{-1}\dfrac{1}{n}\bbU_{1}^{\ast}\bbX\bbX^{\ast}\bbU_{2}\bbD_{2}^{\frac{1}{2}}\left(\dfrac{1}{n}\bbD_{2}^{\frac{1}{2}}\bbU_{2}^{\ast}\bbX\bbX^{\ast}\bbU_{2}\bbD_{2}^{\frac{1}{2}}-z\bbI_{p-M} \right)^{-2}\\
	&\dfrac{1}{n}\bbD_{2}^{\frac{1}{2}}\bbU_{2}^{\ast}\bbX\bbX^{\ast}\bbU_{1}
\end{align*}
We denote the stieltjes transform of $ \bbX^{\ast}\bGma\bbX $ as $ \underline{m}_{2n0} $.
Because
\begin{align} 
	\nonumber&\frac{1}{n}\bbU_{1}^{\ast}\bbX\left(\frac{1}{n}\bbX^{\ast}\bGma\bbX-z\bbI_{n}\right) ^{-1}\bbX^{\ast}\bbU_{1}\\ 
	\nonumber&\approx \dfrac{1}{n}\mathrm{tr}\left(\frac{1}{n} \bbX^{\ast}\bGma\bbX-z\bbI_{n}\right)^{-1}E\left(\frac{1}{n}\bbX\bbX^{\ast}\bbU_{1}\bbU_{1}^{\ast} \right)\\
	\nonumber&=\dfrac{1}{n}\mathrm{tr}\left(\frac{1}{n} \bbX^{\ast}\bGma\bbX-z\bbI_{n}\right)^{-1}\bbI_{p}\\
	&=\underline{m}_{2n0}\left(z \right)\stackrel{d}\longrightarrow \underline{m}\left(z \right), \label{6}
\end{align}
thus
\begin{align} 
	\left(\bbD_{1}^{-1}+\frac{1}{n}\bbU_{1}^{\ast}\bbX\left(\frac{1}{n}\bbX^{\ast}\bGma\bbX-z\bbI_{n}\right) ^{-1}  \bbX^{\ast}\bbU_{1} \right)^{-1}\stackrel{d}\longrightarrow\underline{m}^{-1} \left(z \right).\label{7}
\end{align}
Because
\begin{align} 
	\nonumber&\dfrac{1}{n}\bbU_{1}^{\ast}\bbX\bbX^{\ast}\bbU_{2}\bbD_{2}^{\frac{1}{2}}\left(\dfrac{1}{n}\bbD_{2}^{\frac{1}{2}}\bbU_{2}^{\ast}\bbX\bbX^{\ast}\bbU_{2}\bbD_{2}^{\frac{1}{2}}-z\bbI_{p-M} \right)^{-2}
	\dfrac{1}{n}\bbD_{2}^{\frac{1}{2}}\bbU_{2}^{\ast}\bbX\bbX^{\ast}\bbU_{1}\\
	\nonumber&\approx\dfrac{1}{n}\mathrm{tr}\left(\bbX^{\ast}\bbU_{2}\bbD_{2}^{\frac{1}{2}}\left(\dfrac{1}{n}\bbD_{2}^{\frac{1}{2}}\bbU_{2}^{\ast}\bbX\bbX^{\ast}\bbU_{2}\bbD_{2}^{\frac{1}{2}}-z\bbI_{p-M} \right)^{-2}
	\dfrac{1}{n}\bbD_{2}^{\frac{1}{2}}\bbU_{2}^{\ast}\bbX \right)\\
	\nonumber&=\dfrac{1}{n}\mathrm{tr}\left( \left(\dfrac{1}{n}\bbD_{2}^{\frac{1}{2}}\bbU_{2}^{\ast}\bbX\bbX^{\ast}\bbU_{2}\bbD_{2}^{\frac{1}{2}}-z\bbI_{p-M} \right)^{-2}\dfrac{1}{n}\bbD_{2}^{\frac{1}{2}}\bbU_{2}^{\ast}\bbX\bbX^{\ast}\bbU_{2}\bbD_{2}^{\frac{1}{2}}\right) \\
	\nonumber&=\dfrac{1}{n}\mathrm{tr}\left(  \left( \dfrac{1}{n}\bbD_{2}^{\frac{1}{2}}\bbU_{2}^{\ast}\bbX\bbX^{\ast}\bbU_{2}\bbD_{2}^{\frac{1}{2}}-z\bbI_{p-M} \right)^{-2}\left(  \dfrac{1}{n}\bbD_{2}^{\frac{1}{2}}\bbU_{2}^{\ast}\bbX\bbX^{\ast}\bbU_{2}\bbD_{2}^{\frac{1}{2}}-z\bbI_{p-M}+z\bbI_{p-M} \right)\right) \\
	\nonumber&=\dfrac{1}{n}\mathrm{tr} \left( \dfrac{1}{n}\bbD_{2}^{\frac{1}{2}}\bbU_{2}^{\ast}\bbX\bbX^{\ast}\bbU_{2}\bbD_{2}^{\frac{1}{2}}-z\bbI_{p-M} \right)^{-1}+\dfrac{z}{n}\mathrm{tr} \left( \dfrac{1}{n}\bbD_{2}^{\frac{1}{2}}\bbU_{2}^{\ast}\bbX\bbX^{\ast}\bbU_{2}\bbD_{2}^{\frac{1}{2}}-z\bbI_{p-M} \right)^{-2} \\
	\nonumber&=\dfrac{1}{n}\mathrm{tr}\left(\bbS_{22}-z\bbI_{p-M} \right)^{-1}+\dfrac{z}{n}\mathrm{tr}\left(\bbS_{22}-z\bbI_{p-M} \right)^{-2}\\
	\nonumber&=\dfrac{p-M}{n}\dfrac{1}{p-M} \mathrm{tr}\left(\bbS_{22}-z\bbI_{p-M}\right) ^{-1}+z\dfrac{p-M}{n}\dfrac{1}{p-M} \mathrm{tr}\left(\bbS_{22}-z\bbI_{p-M} \right)^{-2}\\
	&\stackrel{d}\longrightarrow cm\left(z \right)+zcm'\left(z \right)
	=\underline{m}\left(z \right)+z\underline{m}'\left(z \right), \label{8} 
\end{align}
where the last equality is true as $ \underline{m}=-\dfrac{1-c}{z}+cm\left(z \right)  $.\\
Therefore, by (\ref{7}) and (\ref{8}),
\begin{align*}
	T_{2}\stackrel{d}\longrightarrow\dfrac{\underline{m}\left(z \right)+z\underline{m}'\left(z \right)}{z\underline{m}\left(z \right)}M.
\end{align*}
Thus
\begin{align*}
	T_{1}-T_{2}\stackrel{d}\longrightarrow-\dfrac{\underline{m}\left(z \right)+z\underline{m}'\left(z \right)}{z\underline{m}\left(z \right)}M.	
\end{align*}		
Accordingly 
\begin{align*}
	L_{1}-L_{2}\stackrel{d}\longrightarrow\frac{M}{2\pi i}\oint_{\mathcal C}f\left(z \right)\dfrac{\underline{m}(z)+z\underline{m}^{'}(z)}{z\underline{m}(z)}dz=\frac{M}{2\pi i}\oint_{\mathcal C}f\left(z \right)\dfrac{\underline{m}^{'}(z)}{\underline{m}(z)}dz.
\end{align*}
Thus the proof is completed.
\end{proof}

The following lemma derives the explicit asymptotic distribution of  linear spectral statiatics generated from sub-matrix $ \bbS_{22} $.

\begin{lemma}(\cite{10.1214/14-AOS1292})\label{lemma3}
\textit{Define} $ \boldsymbol Q_{n}\left( f_{l}\right)=\sum_{j=1}^{p-M}f_{l}\left( \widetilde{\bgl_{j}}\right)-(p-M)\int f_{l}(x)dF^{c_{nM},H_{2n}},  $
\textit{then the limiting distribution of} $ \left(\boldsymbol Q_{n}\left( f_{1}\right),\cdots,\boldsymbol Q_{n}\left( f_{h}\right) \right) $  \textit{is Gaussian distribution, with mean function}
\begin{align*}
	E\boldsymbol Q_{f}&=
	-\frac{\alpha_{x}}{2 \pi i}\cdot\oint_{\mathcal{C}}f(z)\dfrac{  c \int \underline{m}^{3}(z)t^{2}\left(1+t \underline{m}(z)\right)^{-3} d H(t)}{\left(1-c \int \frac{\underline{m}^{2}(z) t^{2}}{\left(1+t \underline{m}(z)\right)^{2}} d H(t)\right)\left(1-\alpha_{x} c \int \frac{\underline{m}^{2}(z) t^{2}}{\left(1+t \underline{m}(z)\right)^{2}} d H(t)\right) }dz \\
	&-\frac{\beta_{x}}{2 \pi i} \cdot \oint_{\mathcal{C}} f(z) \frac{c \int \underline{m}^{3}(z) t^{2}\left(1+t \underline{m}(z)\right)^{-3} d H(t)}{1-c \int \underline{m}^{2}(z) t^{2}\left(1+t \underline{m}(z)\right)^{-2} d H(t)} dz
\end{align*}
and covariance function
\begin{align*}
	Cov\left(\boldsymbol Q_{f_{m}}, \boldsymbol Q_{f_{l}}\right)
	&=-\frac{1}{4 \pi^{2}} \oint_{\mathcal{C}_{1}} \oint_{\mathcal{C}_{2}} \frac{f_{m}\left(z_{1}\right) f_{l}\left(z_{2}\right)}{\left(\underline{m} {\left.\left(z_{1}\right)-\underline{m}\left(z_{2}\right)\right)^{2}}\right.} d \underline{m}\left(z_{1}\right) d \underline{m}\left(z_{2}\right) 
	-\frac{c \beta_{x}}{4 \pi^{2}} \oint_{\mathcal{C}_{1}} \oint_{\mathcal{C}_{2}} f_{m}\left(z_{1}\right) f_{l}\left(z_{2}\right)\\
	&\left[\int \frac{t}{\left(\underline{m}\left(z_{1}\right) t+1\right)^{2}}\right.
	\left.\times \frac{t}{\left(\underline{m}\left(z_{2}\right) t+1\right)^{2}} d H(t)\right] d \underline{m}\left(z_{1}\right) d \underline{m}\left(z_{2}\right)\\
	&\quad-\frac{1}{4 \pi^{2}} \oint_{\mathcal{C}_{1}} \oint_{\mathcal{C}_{2}} f_{m}\left(z_{1}\right) f_{l}\left(z_{2}\right)\left[\frac{\partial^{2}}{\partial z_{1} \partial z_{2}} \log \left(1-a\left(z_{1}, z_{2}\right)\right)\right] d z_{1} d z_{2}
\end{align*} 
where $ \underline{m}(z) $ is the stieltjes transform of $ \underline{F}^{c,H} $, $\mathcal{C}, \mathcal{C}_{1}$ \textit{and} $\mathcal{C}_{2}$ \textit{are closed contours in the complex plan enclosing the support of the LSD} $ F^{y, H}$, \textit{and} $\mathcal{C}_{1}$ \textit{and }$\mathcal{C}_{2}$ \textit{being nonoverlapping. Finally, the function} $a\left(z_{1}, z_{2}\right)$ \textit{is}
$$
a\left(z_{1}, z_{2}\right)=\alpha_{x}\left(1+\frac{\underline{m}\left(z_{1}\right) \underline{m}\left(z_{2}\right)\left(z_{1}-z_{2}\right)}{\underline{m}\left(z_{2}\right)-\underline{m}\left(z_{1}\right)}\right)
$$
Here $ \beta_{x}= E\left|x_{ij}\right| ^{4}- \left|Ex_{ij}^{2}\right|^{2}-2, \alpha_{x}=\left|Ex_{ij}^{2}\right|^{2} $.
\end{lemma}

\begin{proof}
It is a directly result of \cite{10.1214/14-AOS1292}, thus we do not show the details here.
\end{proof}


The following lemma, borrowing from \cite{10.3150/20-BEJ1237}, characterizes the limiting distribution of spiked eigenvalues of sample covariance matrix.
\begin{lemma} (\citep{10.3150/20-BEJ1237})\label{lemma4}
\textit{ Define random vector} $ \gamma_{k}=\left( \gamma_{kj} \right)' = \left( \sqrt{n}\cdot \dfrac{\lambda_{j}-\phi_{n}\left(\al_{k} \right) }{\phi_{n}\left(\al_{k} \right)},    j\in J_{k} \right)'$ ,  \textit{where} $ \phi_{n}\left( \al_{k}\right) $ \textit{has the same form as} $ \phi_{k} $ \textit{but with} $ c_{n},H_{n} $ \textit{instead of} $ c,H $. \textit{Then random vector} $ \gamma_{k} $ \textit{converges weakly to the joint distribution of $ m_{k} $ eigenvalues of Gaussian random matrix} $$ -\dfrac{1}{\theta_{k}}\left[\bgO_{\phi_{k}} \right]_{kk}   $$
\textit{where} 
$$\theta_{k}=\phi_{k}^{2}\underline{m}_{2}\left( \phi_{k}\right), \underline{m}_{2}\left( \lambda\right)=\int\dfrac{1}{\left( \lambda-x\right) ^{2}}d\underline{F}\left( x\right)   $$ \textit{with} $ \underline{F} $ \textit{being the LSD of matrix} $ n^{-1}\bbX^{\ast}\bGma\bbX,$ $ \phi_{k}=\phi\left(x \right)\mid_{x=\al_{k}}=\al_{k}\left(1+c\int\dfrac{t}{\al_{k}-t}dH\left(t \right)  \right)  $.
$ \left[\bgO_{\phi_{k}} \right]_{kk} $ \textit{is the} $ k $\textit{th diagonal block of matrix} $ \bgO_{\phi_{k}} $. \textit{The variances and covariances of the elements} $ \omega_{ij} $ \textit{of} $ \bgO_{\phi_{k}} $ \textit{is:}
$$
\operatorname{Cov}\left(\omega_{i_{1}, j_{1}}, \omega_{i_{2}, j_{2}}\right)=\left\{\begin{array}{cc}
	(q+1) \theta_{k}+\pi_{x, i i i i} \nu_{k}, & i_{1}=j_{1}=i_{2}=j_{2}=i \\
	\theta_{k}+\pi_{x, i j i j} \nu_{k}, & i_{1}=i_{2}=i \neq j_{1}=j_{2}=j \\
	\pi_{x, i_{1} j_{1} i_{2} j_{2}} \nu_{k}, & \text { other cases }
\end{array}\right.
$$
\textit{where} $\pi_{x, i_{1} j_{1} i_{2} j_{2}}=\lim \sum_{t=1}^{p} \bar{u}_{t i_{1}} u_{t j_{1}} u_{t i_{2}} \bar{u}_{t j_{2}} \mathrm{E}\left\{\left|x_{11}\right|^{4} I\left(\left|x_{11}\right| \leq \sqrt{n}\right)-2-q\right\}$, $\mathbf{u}_{i}=\left(u_{1 i}, \ldots, u_{p i}\right)^{\prime}$ \textit{are the} $i$ \textit{th column of the matrix} $\mathbf{U}_{1}$, $\nu_{k}=\phi_{k}^{2} \underline{m}^{2}\left(\phi_{k}\right)$, \textit{and} $ q=1 $ \textit{for the real case}, $ q=0 $ \textit{for the complex case}.
\end{lemma}

Recall $ \lambda_{j} $ is the eigenvalue of $ \bbB $, $ \widetilde{\bgl_{j}} $ is the eigenvalue of $ \bbS_{22} $. The following lemma shows the independence between $ \sum_{j=1}^{M}f\left( \lambda_{j}\right) $ and $ \sum_{j=1}^{p-M}f\left(\widetilde{\bgl_{j}} \right) $.

\begin{lemma}\label{lemma5}
\textit{The spiked linear spectral statistics} $ \sum_{j=1}^{M}f\left( \lambda_{j}\right) $ \textit{and the bulk linear spectral statistics} $ \sum_{j=1}^{p-M}f\left(\widetilde{\bgl_{j}} \right) $ \textit {are asymptotically independent.}
\end{lemma}
\begin{proof}
Firstly we consider the indenpendence between $ \sum_{j=1}^{M}f\left( \lambda_{j}\right) $  and $ \sum_{j=M+1}^{p}f\left( \lambda_{j}\right) $. If we can prove that given $ \sum_{j=M+1}^{p}f\left( \lambda_{j}\right) $, the asymptotic limiting distribution of the spiked part $ \sum_{j=1}^{M}f\left( \lambda_{j}\right) $ only depends on the bulk part $ \sum_{j=M+1}^{p}f\left( \lambda_{j}\right) $ limiting result, which means it does not depend on the random part of the bulk part, thus we can get the indenpendence of the spiked part and bulk part. Without less of generality, we consider $ f\left(x \right)=x $ first.  \\
When $ f\left(x \right)=x $, from proof of Theorem 3.2 in \cite{10.3150/20-BEJ1237}, they get an equation $$ 0=\left| \left[\bgO_{M}\left( \phi_{k}\right) \right] _{kk}+\mathrm{ lim}\gamma_{kj}\left\lbrace \phi_{k}^{2}\underline m_{2}(\phi_{k})\right\rbrace \bbI_{m_{k}}\right|$$ where $\left[\bgO_{M}\left( \phi_{k}\right) \right] =\dfrac{\phi_{k}}{\sqrt{n}}\left[ \mathrm{tr}\left\lbrace \left(  \phi_{k}\bbI-\frac{1}{n}\bbX^{*}\bGma \bbX\right) ^{-1}\right\rbrace \bbI-\bbU_{1}^{*}\bbX\left(\phi_{k}\bbI-\frac{1}{n}\bbX^{*}\bGma \bbX\right) ^{-1}\bbX^{*}\bbU_{1}\right] $,  $\left[\cdot  \right]_{kk} $ is the $ k $ th diagonal element of a matrix. $ \bGma=\bbU_{2}\bbD_{2}\bbU_{2}^{*}$. $ \gamma_{kj}=\sqrt{n}\left( \dfrac{\lambda_{j}}{\phi_{n}\left(\al_{k} \right) }-1 \right)$. $\underline m_{2}(\phi_{k}) $ is the limit of $ \mathrm{tr}\left\lbrace \left(  \phi_{k}\bbI-\frac{1}{n}\bbX^{*}\bGma \bbX\right) ^{-2}\right\rbrace $.  From the equation, we know that $ \gamma_{kj} $ has the same asymptotic limiting distribution with
$ -\dfrac{\left[\bgO_{M}\left( \phi_{k}\right) \right] _{kk}}{\phi_{k}^{2}\underline m_{2}(\phi_{k})} $, given $\mathrm{tr}\left\lbrace \left(  \phi_{k}\bbI-\frac{1}{n}\bbX^{*}\bGma \bbX\right) ^{-1}\right\rbrace $, the limiting distribution of $ -\dfrac{\left[\bgO_{M}\left( \phi_{k}\right) \right] _{kk}}{\phi_{k}^{2}\underline m_{2}(\phi_{k})} $ only depends on the bulk part limiting result, has nothing to do with the random part, so the indenpendence between the spiked part and bulk part is obtained when $ f\left(x \right)= x $. \\
When $ f\left(x \right)\neq x $, by Newton-Leibniz formula, we have 
\begin{align*}
	&\sum_{j=1}^{M}f\left(\lambda_{j}\right)-\sum_{k=1}^{K}\sum_{j\in J_{k}}f\left(\phi_{n}\left(\bbalp_{j} \right) \right)\\
	&=\sum_{k=1}^{K}\sum_{j\in J_{k}}\int_{0}^{\dfrac{\phi_{n}\left(\bbalp_{j} \right)}{\sqrt{n}}\gamma_{kj}}f'\left(t+\phi_{n}\left(\bbalp_{j} \right) \right)dt\\
	&=\sum_{k=1}^{K}\sum_{j\in J_{k}}\int_{0}^{1}\dfrac{\phi_{n}\left(\bbalp_{j} \right)}{\sqrt{n}}\gamma_{kj}f'\left(\phi_{n}\left(\bbalp_{j} \right)\left(1+\frac{\gamma_{kj}}{\sqrt{n}}s\right)  \right)ds\\
	&=\sum_{k=1}^{K}\sum_{j\in J_{k}}\int_{0}^{1}\dfrac{\phi_{n}\left(\bbalp_{j} \right)}{\sqrt{n}}\gamma_{kj}\dfrac{f'\left(\phi_{n}\left(\bbalp_{j} \right)\left(1+\frac{\gamma_{kj}}{\sqrt{n}}s\right)  \right)}{f'\left(\phi_{n}\left(\bbalp_{j} \right)\right) }f'\left( \phi_{n}\left(\bbalp_{j} \right)\right)ds\\
	&\stackrel{*}\rightarrow\sum_{k=1}^{K}\sum_{j\in J_{k}}\int_{0}^{1}\dfrac{\phi_{n}\left(\bbalp_{j} \right)}{\sqrt{n}}\gamma_{kj}f'\left( \phi_{n}\left(\bbalp_{j} \right)\right)ds
	=\sum_{k=1}^{K}\sum_{j\in J_{k}}\dfrac{\phi_{n}\left(\bbalp_{j} \right)}{\sqrt{n}}\gamma_{kj}f'\left( \phi_{n}\left(\bbalp_{j} \right)\right)    
\end{align*}
where $ (\ast) $ is true due to Assumption 6. Thus we turn it into a function of $ \gamma_{kj} $, we have prove above that given non-spiked eigenvalues, the limiting distribution of $ \gamma_{kj} $ is only concerned with bulk part's limiting result, so does $ \sum_{k=1}^{K}\sum_{j\in J_{k}}\dfrac{\phi_{n}\left(\bbalp_{j} \right)}{\sqrt{n}}\gamma_{kj}f'\left( \phi_{n}\left(\bbalp_{j} \right)\right) $, accordingly we could get the conclusion that  $ \sum_{j=1}^{M}f\left( \lambda_{j}\right) $  and $ \sum_{j=M+1}^{p}f\left( \lambda_{j}\right) $ are mutually asymptotic independent.\\
From Lemma \ref{lemma2} we know the difference between $\sum_{j=M+1}^{p}f\left( \lambda_{j}\right) $ and $ \sum_{j=1}^{p-M}f\left( \widetilde{\bgl_{j}}\right)$ tends to a nonrandom part $ -\frac{M}{2\pi i}\oint_{C}f\left(z \right)\dfrac{\underline{m}^{'}(z)}{\underline{m}(z)}dz $. So $ \sum_{j=1}^{M}f\left( \lambda_{j}\right) $ is also asymptotic independent of $ \sum_{j=1}^{p-M}f\left( \widetilde{\bgl_{j}}\right)$, the proof is accomplished.
\end{proof}

\subsection{Proof of Theorem \ref{thm1}}

The proof of Theorem \ref{thm1} builds on the decomposition analysis of linear spectral statistics, dividing it into the spiked part $ \sum_{j=1}^{M}f\left( \bgl_{j}\right) $ and the bulk part $ \sum_{j=M+1}^{p}f\left( \bgl_{j}\right) $. We focus on each dimension of the random vector at first,
so we have
\begin{align*}
	&\sum_{j=1}^{p}f\left( \bgl_{j}\right)-p\int f(x)dF^{c_{n},H_{n}} \\
	&=\sum_{j=1}^{M}f\left( \bgl_{j}\right)+\sum_{j=M+1}^{p}f\left( \bgl_{j}\right)-p\int f(x)dF^{c_{n},H_{n}}\\
	&=\sum_{j=1}^{M}f\left( \bgl_{j}\right)+\sum_{j=1}^{p-M}f\left( \widetilde{\bgl_{j}}\right)-(p-M)\int f(x)dF^{c_{nM},H_{2n}}+\sum_{j=M+1}^{p}f\left( \bgl_{j}\right)-\sum_{j=1}^{p-M}f\left( \widetilde{\bgl_{j}}\right)\\
	&+(p-M)\int f(x)dF^{c_{nM},H_{2n}}-p\int f(x)dF^{c_{n},H_{n}}
\end{align*}
Lemma \ref{lemma1} has shown the difference between $ (p-M)\int f(x)dF^{c_{nM},H_{2n}}-p\int f(x)dF^{c_{n},H_{n}}$ is 0. 
Moreover, in Lemma \ref{lemma2} we have proved $$ \sum_{j=M+1}^{p}f\left( \bgl_{j}\right)-\sum_{j=1}^{p-M}f\left( \widetilde{\bgl_{j}}\right)=\frac{M}{2\pi i}\oint_{C}f\left(z \right)\dfrac{\underline{m}^{'}(z)}{\underline{m}(z)}dz. $$
Thus 
\begin{align*}
	&\sum_{j=1}^{p}f\left( \bgl_{j}\right)-p\int f(x)dF^{c_{n},H_{n}}\\
	&=\sum_{j=1}^{M}f\left( \bgl_{j}\right)+\sum_{j=1}^{p-M}f\left( \widetilde{\bgl_{j}}\right)-(p-M)\int f(x)dF^{c_{nM},H_{2n}}+\frac{M}{2\pi i}\oint_{C}f\left(z \right)\dfrac{\underline{m}^{'}(z)}{\underline{m}(z)}dz,
\end{align*}
which then yields
\begin{align}
	&\sum_{j=1}^{p}f\left( \bgl_{j}\right)-p\int f(x)dF^{c_{n},H_{n}}-\sum_{k=1}^{K}\sum_{j\in J_{k}}f\left( \phi_{n}\left(\al_{j} \right) \right)-\frac{M}{2\pi i}\oint_{C}f\left(z \right)\dfrac{\underline{m}^{'}(z)}{\underline{m}(z)}dz\label{101}\\
	&=\sum_{j=1}^{M}f\left( \bgl_{j}\right)-\sum_{k=1}^{K}\sum_{j\in J_{k}}f\left( \phi_{n}\left(\al_{j} \right) \right)+\sum_{j=1}^{p-M}f\left( \widetilde{\bgl_{j}}\right)-(p-M)\int f(x)dF^{c_{nM},H_{2n}}. \label{100}
\end{align}
The analysis below is carried out by dividing the (\ref{100}) into two parts: the spiked part $ \sum_{j=1}^{M}f\left( \bgl_{j}\right)-\sum_{k=1}^{K}\sum_{j\in J_{k}}f\left( \phi_{n}\left(\al_{j} \right) \right) $ and the bulk part $ \sum_{j=1}^{p-M}f\left( \widetilde{\bgl_{j}}\right)-(p-M)\int f(x)dF^{c_{nM},H_{2n}} $.

We consider the asymptotic distribution of the bulk part first. It is a direct result of \cite{10.1214/aop/1078415845} or \cite{10.1214/14-AOS1292}. Therefore from Lemma \ref{lemma3} the limiting distribution of  $$\left(  \sum_{j=1}^{p-M}f_{1}\left( \widetilde{\bgl_{j}}\right)-(p-M)\int f_{1}(x)dF^{c_{nM},H_{2n}},\cdots,\sum_{j=1}^{p-M}f_{h}\left( \widetilde{\bgl_{j}}\right)-(p-M)\int f_{h}(x)dF^{c_{nM},H_{2n}} \right)  $$ is Gaussian distribution, with mean function
\begin{align*}
	\mu\triangleq E\boldsymbol Q_{f}&=
	-\frac{\alpha_{x}}{2 \pi i}\cdot\oint_{\mathcal{C}}f(z)\dfrac{  c \int \underline{m}^{3}(z)t^{2}\left(1+t \underline{m}(z)\right)^{-3} d H(t)}{\left(1-c \int \frac{\underline{m}^{2}(z) t^{2}}{\left(1+t \underline{m}(z)\right)^{2}} d H(t)\right)\left(1-\alpha_{x} c \int \frac{\underline{m}^{2}(z) t^{2}}{\left(1+t \underline{m}(z)\right)^{2}} d H(t)\right) }dz \\
	&-\frac{\beta_{x}}{2 \pi i} \cdot \oint_{\mathcal{C}} f(z) \frac{c \int \underline{m}^{3}(z) t^{2}\left(1+t \underline{m}(z)\right)^{-3} d H(t)}{1-c \int \underline{m}^{2}(z) t^{2}\left(1+t \underline{m}(z)\right)^{-2} d H(t)} dz
\end{align*}
and covariance function
\begin{align*}
	\sigma_{s,t}^{2}\triangleq Cov\left(\boldsymbol Q_{f_{s}}, \boldsymbol Q_{f_{t}}\right)
	&=-\frac{1}{4 \pi^{2}} \oint_{\mathcal{C}_{1}} \oint_{\mathcal{C}_{2}} \frac{f_{s}\left(z_{1}\right) f_{t}\left(z_{2}\right)}{\left(\underline{m} {\left.\left(z_{1}\right)-\underline{m}\left(z_{2}\right)\right)^{2}}\right.} d \underline{m}\left(z_{1}\right) d \underline{m}\left(z_{2}\right) 
	-\frac{c \beta_{x}}{4 \pi^{2}} \oint_{\mathcal{C}_{1}} \oint_{\mathcal{C}_{2}} f_{s}\left(z_{1}\right) f_{t}\left(z_{2}\right)\\
	&\left[\int \frac{t}{\left(\underline{m}\left(z_{1}\right) t+1\right)^{2}}\right.
	\left.\times \frac{t}{\left(\underline{m}\left(z_{2}\right) t+1\right)^{2}} d H(t)\right] d \underline{m}\left(z_{1}\right) d \underline{m}\left(z_{2}\right)\\
	&\quad-\frac{1}{4 \pi^{2}} \oint_{\mathcal{C}_{1}} \oint_{\mathcal{C}_{2}} f_{s}\left(z_{1}\right) f_{t}\left(z_{2}\right)\left[\frac{\partial^{2}}{\partial z_{1} \partial z_{2}} \log \left(1-a\left(z_{1}, z_{2}\right)\right)\right] d z_{1} d z_{2}
\end{align*} 
where $ \underline{m}(z) $ is the stieltjes transform of $ \underline{F}^{c,H} $, $\mathcal{C}, \mathcal{C}_{1}$ and $\mathcal{C}_{2}$ are closed contours in the complex plan enclosing the support of the LSD $ F^{y, H}$, and $\mathcal{C}_{1}$ and $\mathcal{C}_{2}$ being nonoverlapping, the function $a\left(z_{1}, z_{2}\right)$ is
$$
a\left(z_{1}, z_{2}\right)=\alpha_{x}\left(1+\frac{\underline{m}\left(z_{1}\right) \underline{m}\left(z_{2}\right)\left(z_{1}-z_{2}\right)}{\underline{m}\left(z_{2}\right)-\underline{m}\left(z_{1}\right)}\right).
$$
Here $ \beta_{x}= E\left|x_{ij}\right| ^{4}- \left|Ex_{ij}^{2}\right|^{2}-2, \alpha_{x}=\left|Ex_{ij}^{2}\right|^{2} $.

Then we consider the asymptotic distribution of the spiked part $ \sum_{j=1}^{M}f\left(\lambda_{j}\right)-\sum_{k=1}^{K}\sum_{j\in J_{k}}f\left(\phi_{n}\left(\bbalp_{j} \right) \right) $, from the proof of Lemma \ref{lemma5}, $ \sum_{j=1}^{M}f\left( \bgl_{j}\right)-\sum_{k=1}^{K}\sum_{j\in J_{K}}f\left(\phi_{n}\left(\bbalp_{j} \right)\right)  $ has the same limiting distribution as $ \sum_{k=1}^{K}\dfrac{\phi_{n}\left(\bbalp_{k} \right)}{\sqrt{n}}f'\left( \phi_{n}\left(\bbalp_{k} \right)\right)\sum_{j\in J_{k}}\gamma_{kj}  $.  

From Lemma \ref{lemma4} we have $ \left(\gamma_{kj}, j\in J_{k} \right)'\stackrel{d}\rightarrow-\dfrac{1}{\theta_{k}}\left[\bgO_{\phi_{k}} \right]_{kk},  $  so $ \sum_{j\in J_{k}}\gamma_{kj}\stackrel{d}\rightarrow -\dfrac{1}{\theta_{k}}\mathrm {tr}\left[\bgO_{\phi_{k}} \right]_{kk}    $. Recall that $ \omega_{ij} $ is the elements of $ \bgO_{\phi_{k}} $, $ \mathrm {tr}\left[\bgO_{\phi_{k}} \right]_{kk} $ is the summation of the diagonal element, that is $ \sum_{j\in J_{k}} \omega_{jj} $. Because the diagonal elements are i.i.d., thus $ E\left(\sum_{j\in J_{k}} \omega_{jj}\right) =0  $, $$ Var\left(\sum_{j\in J_{k}} \omega_{jj}\right)=\sum_{j\in J_{k}}Var\left(\omega_{jj} \right)+ \sum_{j_{1}\neq j_{2}}cov\left(\omega_{j_{1}j_{1}}, \omega_{j_{2}j_{2}} \right)= \sum_{j\in J_{k}}\left( \left(q+1 \right)\theta_{k}+\pi_{x,jjjj}\nu_{k}\right) +\sum_{j_{1}\neq j_{2}}\pi_{x,j_{1}j_{1}j_{2}j_{2}}\nu_{k}  $$ Therefore
from Lemma \ref{lemma4} we have, the asymptotic distribution of $ \sum_{k=1}^{K}\sum_{j\in J_{k}}\gamma_{kj} $ is Gaussian distribution, with $$
E\left(\sum_{k=1}^{K}\sum_{j\in J_{k}}\gamma_{kj}\right)=0,\\
\sigma_{k}^{2}\triangleq Var\left(\sum_{k=1}^{K}\sum_{j\in J_{k}}\gamma_{kj}\right)=\sum_{k=1}^{K}\dfrac{ \sum_{j\in J_{k}}\left( \left(q+1 \right)\theta_{k}+\pi_{x,jjjj}\nu_{k}\right) +\sum_{j_{1}\neq j_{2}}\pi_{x,j_{1}j_{1}, j_{2}j_{2}}\nu_{k}}{\theta_{k}^{2}} $$ 

Finally we focus on the asymptotic distribution of the equation \ref{100}, for simplicity we define $$\rho_{l}=\dfrac{1}{\sqrt{\sum_{k=1}^{K}\left(\dfrac{\phi_{n}\left(\al_{k} \right) }{\sqrt{n}}f'_{l}\left(\phi_{n}\left(\al_{k} \right) \right)  \right)^{2}+1 }},$$
$$ G_{n}\left( x\right)=p\left[F^{\bbB}\left(x \right)-F^{c_{n},H_{n}}\left(x \right)   \right],  $$ 
$$ \boldsymbol Y_{n}\left(f_{l} \right)= \rho_{l}\left( \int f_{l}\left(x \right)dG_{n}\left( x\right)-\sum_{k=1}^{K}\sum_{j\in J_{k}}\left(\phi_{n}\left(\al_{j} \right)  \right)-\dfrac{M}{2\pi i}\oint_{\mathcal C}f_{l}\left(z \right)\dfrac{\underline{m}'(z)}{\underline{m}(z)}dz\right). $$ 
the random vector $\left( \boldsymbol Y_{n}\left(f_{1} \right) ,\cdots,\boldsymbol Y_{n}\left(f_{h} \right) \right)  $ could be written into the summation of two random vectors: the spiked part $ \left(\rho_{1} \sum_{k=1}^{K}\dfrac{\phi_{n}\left(\bbalp_{k} \right)}{\sqrt{n}}f_{1}'\left( \phi_{n}\left(\bbalp_{k} \right)\right)\sum_{j\in J_{k}}\gamma_{kj} ,\cdots,\rho_{h}\sum_{k=1}^{K}\dfrac{\phi_{n}\left(\bbalp_{k} \right)}{\sqrt{n}}f_{h}'\left( \phi_{n}\left(\bbalp_{k} \right)\right)\sum_{j\in J_{k}}\gamma_{kj}  \right)$ and the bulk part $\left(  \rho_{1}\left[ \sum_{j=1}^{p-M}f_{1}\left( \widetilde{\bgl_{j}}\right)-(p-M)\int f_{1}(x)dF^{c_{nM},H_{2n}}\right] ,\cdots,\rho_{h}\left[ \sum_{j=1}^{p-M}f_{h}\left( \widetilde{\bgl_{j}}\right)-(p-M)\int f_{h}(x)dF^{c_{nM},H_{2n}}\right] \right)    $
From the result above, the mean function of the spiked part is 0 and the covariance function is 
$$
Cov\left(\boldsymbol Y_{f_{s}}, \boldsymbol Y_{f_{t}} \right)=\rho_{s}\rho_{t}\sum_{k=1}^{K}\dfrac{\phi_{n}^{2}\left(\bbalp_{k} \right)}{n}f_{s}'\left( \phi_{n}\left(\bbalp_{k} \right)\right)f_{t}'\left( \phi_{n}\left(\bbalp_{k} \right)\right)\sigma_{k}^{2},   
$$
and the mean function of the bulk part is $ \rho_{l}\mu $ and the covariance function is 
$
 \rho_{s} \rho_{t}\sigma_{s,t}^{2}.
$ Because of Lemma \ref{lemma5}, the spiked part and the bulk part are asymptotically independent, thus $$E\left( \boldsymbol Y_{f_{l}}\right)= \rho_{l}\mu, Cov\left(\boldsymbol Y_{f_{s}}, \boldsymbol Y_{f_{t}} \right)=\rho_{s}\rho_{t}\sum_{k=1}^{K}\dfrac{\phi_{n}^{2}\left(\bbalp_{k} \right)}{n}f_{s}'\left( \phi_{n}\left(\bbalp_{k} \right)\right)f_{t}'\left( \phi_{n}\left(\bbalp_{k} \right)\right)\sigma_{k}^{2}+\rho_{s} \rho_{t}\sigma_{s,t}^{2}.  $$ Therefore the proof is finished.

\bibliography{Unbounded_norm}

\end{document}